\documentclass[11pt,twoside]{amsart}
\usepackage{amsmath}
\usepackage{amssymb}
\usepackage{amsfonts}
\usepackage{amscd}
\usepackage{color}

\newtheorem{thm}{Theorem}[section]

\newtheorem{lemma}[thm]{Lemma}

\newtheorem{cor}[thm]{Corollary}

\theoremstyle{definition}

\newtheorem{rem}[thm]{Remark}

\newcommand\op{\operatorname}

\newcommand\Del{\op{Del}}

\newcommand\diag{\op{diag}}

\newcommand\GL{\op{GL}}

\newcommand\Semi{\op{Semi}\,}

\newcommand\Spec{\op{Spec}\,}

\newcommand{\wcL}{{\widetilde {\cal L}}}

\newcommand{\wP}{{\widetilde P}}
\newcommand{\wQ}{{\widetilde Q}}

\newcommand{\bQ}{{\bold Q}}
\newcommand{\bR}{{\bold R}}

\newcommand{\bZ}{{\bold Z}}

\newcommand{\XR}{X_{\bold R}}

\newcommand{\cD}{{\cal D}}

\newcommand{\cL}{{\cal L}}

\ifx\undefined\pf
\newcommand{\pf}{\proof}
\fi
\ifx\undefined\cal
\newcommand{\cal}{\mathcal}
\fi
\ifx\undefined\Bbb
\newcommand{\Bbb}{\mathbb}
\fi
\ifx\undefined\bold
\newcommand{\bold}{\mathbf}
\fi
\ifx\undefined\frak
\newcommand{\frak}{\mathfrak}
\fi
\begin{document}
\title[Delaunay decompositions in dimension 4]
{Delaunay decompositions in dimension 4}
\author{Iku Nakamura and Ken Sugawara}
\email{nakamura@math.sci.hokudai.ac.jp\\ 
sugawark@sap.hokkyodai.ac.jp}
\thanks{The first author is supported in part by the Grant-in-aid 
(No. 17K05188) for Scientific Research, 
JSPS}
\thanks{2000 {\it Mathematics Subject Classification}.
Primary 14J10;Secondary 14K10, 14K25.}
\thanks{{\it Key words and phrases}. Abelian variety, 
Delaunay decomposition, Stable quasi-abelian variety,  Voronoi cone, 
Voronoi compactification}
\date{\today}
%
\begin{abstract}
We prove that any Delaunay decomposition in dimension 4 
is simplicially generating. 
\end{abstract}
\maketitle
\setcounter{section}{-1}

\section{Introduction}
The Voronoi cone decompositions has been attracting our attention in the compactification problem of the moduli scheme of abelian 
varieties since Namikawa's work \cite{Namikawa76}. 
The objects to add as the boundary of the moduli scheme are stable quasi-abelian schemes, reduced or nonreduced, which are described in terms of Delaunay decompositions \cite{Nakamura75}, \cite{Namikawa76}, \cite{AN99}, \cite{Nakamura99}, 
\cite{Nakamura10}. The purpose of this note is 
to prove  that any Delaunay decomposition is simplicially generating 
if the dimension is less than 5 (Theorem~\ref{thm:nilpotency}).  As a corollary, if the dimension is less than 5, a kind of stable quasi-abelian schemes called  projectively stable quasi-abelian schemes 
are reduced, and therefore two kinds of stable quasi-abelian schemes,  
projectively stable quasi-abelian schemes  \cite{Nakamura99} and 
torically stable quasi-abelian varieties \cite{Nakamura10},  
are the same class of varieties.   
\section{Delaunay decompositions}\label{sec:Delaunay decomposition}
\subsection{Definitions and Notation}
\label{defn:Delaunay cell} 
Let $\bZ_0$ (resp. $\bQ_0$, $\bR_0$) be 
the set of all nonnegative integers (resp. 
nonnegative rational numbers, 
nonnegative real numbers). 
Let $X$ be a lattice of rank $g$, 
$X_{\bR}=X\otimes\bR$,
and let $B:X\times X\to \bZ$ be 
a real positive definite symmetric bilinear form $B:X\times X\to \bZ$, 
which determines the inner product 
$B(\phantom{a}, \phantom{a})$ 
and a distance 
 $\|\ \|$ on the Euclidean space $X_{\bR}$ by  
$\|x\|:=\sqrt{B(x,x)}$ $(x\in X_{\bR})$ respectively. 
In what follows we fix $B$ once for all. 
For any $\alpha\in X_{\bR}$ we say 
that $a\in X$ is nearest to $\alpha$ if
\begin{equation*}
\|a-\alpha\|= \min\{ \|b-\alpha\|; b\in X \} \end{equation*}

We define a (closed) {\em $B$-Delaunay cell\/ $\sigma$} 
(or simply a D-cell if $B$ is understood)
 to be the closed convex closure of all lattice elements 
which are nearest to $\alpha$ for some $\alpha\in\XR$. 
Note that for a given D-cell $\sigma$,  $\alpha$ 
is uniquely defined only 
if $\sigma$ has the maximal possible dimension, 
equal to $g$. In this case we call $\alpha$  
the {\em hole} or the {\em center} of $\sigma$.
Together all the $B$-D-cells constitute 
a locally finite decomposition of $\XR$ into infinitely many 
bounded convex polyhedra which we call the {\em $B$-Delaunay 
decomposition $\Del_B$.\/}
It is clear from the definition 
that the Delaunay decomposition is invariant 
under translation by the lattice $X$ and 
that the 0-dimensional cells are precisely 
the elements of $X$. 
Let $\Del:=\Del_B$, and 
$\Del(c)$ the set of all the D-cells containing $c\in X$. 
\par
 For any closed subset $S$ of $\XR$, we define
$\langle S\rangle$ be the convex closure of $S$. 
If $0\in S$, then we define $C(0,S)$ (resp. $\Semi(0,S)$) 
to be the closed cone 
generated  by $S$ over $\bR_0$ 
(resp. the semigroup generated by $S$ over $\bZ_0$).  \par 
A closed convex bounded subset $\tau$ of $\XR$ 
 is called {\em integral} if $\tau$ is 
the convex closure of $\tau\cap X$. 
So any D-cell is integral.  
Let $\tau$ be a $g$-dimensional integral (closed convex) 
subset of $\XR$ which contains the origin $0$.  Then $\tau$  
is called 
{\it totally generating} if $C(0,\tau)\cap X=\Semi(0,\tau\cap X)$,   
while $\tau$ is called {\it simplicially generating} if there exist 
subsets $\tau_i$ $(i\in I)$
of $\tau$:
\begin{enumerate}
\item[$-$] $0\in\tau_i$ and $\tau_j$
have no common interiors for any $i\neq j$;
\item[$-$] $C(0,\tau_i)\cap X=\Semi(0,\tau_i\cap X)$;
\item[$-$] $C(0,\tau)=\cup_{i\in I}C(0,\tau_i)$.
\end{enumerate}

We need to modify the notion ''simplicially generating'' in \cite[1.1, pp.~662-663]{Nakamura99} into the above in order to prove Theorem~\ref{thm:nilpotency} because there might exist $\tau_i$ such that $0\not\in\tau_i$. See  \cite[1.6, pp.~662-663]{Nakamura99} and Remark.~\ref{rem:sigma4 does not contain 0}.

\subsection{The Voronoi cones}
\label{subsec:Voronoi cones}
Let $X$ be a lattice of rank $g$, and $Y_g$ (resp. $Y_g^+$)  the space 
of real positive semi-definite symmetric bilinear forms 
on $\XR$. Let $s_i$ $(i\in [1,g])$ be a $\bZ$-basis of $X$. 
Then we identify $Y_g$ with 
the space of real positive semi-definite symmetric $g\times g$ matrices $B$ 
in an obvious manner: $B=(b_{ij})=(B(s_i,s_j))$. For a subset $I$ of $[1,g]$, we set $s_I=\sum_{i\in I}s_i$. 
\par
A closed subset $V$  of $Y_g$ is called a {\em Voronoi cone} 
if there exists a unique 
Delaunay decomposition of $\XR$, which we denote by  $\Del_V$, such that 
\begin{enumerate}
\item[(a)] $\Del_V=\Del_B$ for any $B\in V^0$;
\item[(b)] for $B\in Y_g$, 
$B\in V^0$ iff $\Del_{B}=\Del_V$,
\end{enumerate}where $V^0$ stands for the relative interior of $V$. 
The Vononoi cones form a closed cone decomposition of $Y_g$,  
which is called the {\em (second) Voronoi decomposition} of $Y_g$. 
By \cite{Voronoi09}, the  {\em (second) Voronoi decomposition} of $Y_g$ is admissible \cite[p.~252]{AMRT75}, \cite[2.4]{Namikawa76}. 
\par
The following is known by \cite[pp.~157-175]{Voronoi09}.
\begin{enumerate}
\item[(i)] For $g\leq 3$, there is a unique $\GL(g,\bZ)$-equivalence class $V$
of $\frac{1}{2}g(g+1)$-dimensional Voronoi cone, and 
any g-dimensional D-cell of $\Del_{V}(0)$ is  
a simplex with $g$ vertices besides 0, which is simplicially generating.
\item[(ii)] There are only three $\GL(4,\bZ)$-equivalence classes of $10$-dimensional Voronoi cones $V_1$, $V_2$ and $V_3$.  
See Subsec.~\ref{subsec:notation} for $V_i$.
\item[(iii)] Any 4-dimensional D-cell of $\Del_{V_i}(0)$ $(i=1,2,3)$ is  
a simplex with 4 vertices besides 0, which is simplicially generating.   
\end{enumerate}

A proof of (iii) is outlined in Sec.~\ref{sec:Delaunay in dim 4}.  
See Theorem~\ref{thm:Vor dim 4}

\subsection{Fusion of Delaunay $g$-cells}
Let $V$ be a Voronoi cone, $V^0$ the relative interior of $V$, 
and $V'$ a face of $V$, 
that is, a Voronoi cone contained in $V$ 
which is a proper subset of $V$.  
Let $B:[0,1]\to Y_g$ be a continuous map 
such that $B((0,1])\subset V^0$ and $B(0)\in V'$. 
Let $\Del_{V}$ (resp. $\Del_{V'}$) 
be the Delaunay decomposition of $V$ (resp. $V'$). 

The following may be rather implicit in \cite{Voronoi09}, \cite{ER88} and 
\cite{Val03}.
\begin{lemma}
\label{lemma:fusion}Let $V$ and $V'$ be Voronoi cones  
such that $V'\subset V$.  
Then any Delaunay $g$-cell of $V$ is contained 
in some Delaunay $g$-cell of $V'$.
Conversely any Delaunay $g$-cell of $V'$ is the   
union of some Delaunay $g$-cells of $V$.
\end{lemma} 
\begin{proof}The Delaunay decomposition of $V$ (resp. $V'$) 
 is the same as that of $B(t)$ $(t\in (0,1]$  (resp. that of $B(0)$), 
Let $\sigma$ be a Delaunay $g$-cell of $V$, and 
$\alpha$ the center of $\sigma$. Though $\sigma$ is constant, 
the center $\alpha=\alpha(t)$ 
is a continuous function of $t\in (0,1]$, which is a polynomial 
function of the coefficients of $B(t)$ divided by $|B(t)|$. 
Since $|B(0)|\neq 0$, $\alpha(0):=\lim_{t\to 0}\alpha(t)$ exists. 
Since there is no element of $X$ in the interior of $\sigma$, 
$\sigma\cap X$ is the set of all vertices of $\sigma$.
Let $\sigma\cap X:=\{a_i; i\in [1,n]\}$. 
For $t\in (0,1]$,  
$$B(t)(a_i-\alpha(t),a_i-\alpha(t))=\min_{a\in X}B(t)(a-\alpha(t),a-\alpha(t)).$$ Hence 
$$B(0)(a_i-\alpha(0),a_i-\alpha(0))=\min_{a\in X}B(0)(a-\alpha(0),a-\alpha(0)).$$ This implies that $\alpha(0)$ is a center of a Delaunay cell $\sigma'$ which contains $\sigma$. Since $\XR$ is the union of Delaunay cells of $B(t)$, 
this shows that any Delaunay $g$-cell of $V'$ is the  
union of some Delaunay $g$-cells of $V$. 
\end{proof}

\begin{cor}
\label{cor:generation}
Let $\sigma'\in\Del_{V'}(0)$ 
such that $\sigma'=\cup_{\lambda\in\Lambda}\sigma_{\lambda}$ 
for some Delaunay cells $\sigma_{\lambda}\in\Del_{V}$. Then 
\begin{enumerate}
\item $\displaystyle{C(0, \sigma',0)=\bigcup_{0\in\sigma_{\lambda} \subset\sigma'}C(0, \sigma_{\lambda}).}$ 
\item If $C(0, \sigma_{\lambda})\cap X=\Semi(0, \sigma_{\lambda}\cap X)$ for any $\lambda$, then $C(0, \sigma')\cap X=\Semi(0,\sigma'\cap X)$.
\end{enumerate}
\end{cor}
\begin{proof}Let $D(r)$ be a ball of radius $r$ with the origin as center.  
Since $\sigma'$ is convex, $C(0,\sigma')=C(0,D(r)\cap\sigma')$ for a sufficiently small $r>0$. By Lemma~\ref{lemma:fusion}, 
$$C(0,D(r)\cap\sigma')=C(0,D(r)\cap\cup_{\lambda\in\Lambda}\sigma_{\lambda})
=\bigcup_{0\in\sigma_{\lambda}}C(0,\sigma_{\lambda}),$$ 
where we caution that there might exist $\lambda\in\Lambda$ such that 
$0\not\in\sigma_{\lambda}$. This proves (1). 
\footnote{Not all $\sigma_{\lambda}$ contains $0$. 
See Remark~\ref{rem:sigma4 does not contain 0}. }
(2) follows from (1). 
\end{proof}

\subsection{Simplicial generation}
The purpose of this note is to prove the following 
theorem, which seems to be known but only implicit in \cite{Voronoi09}:
\begin{thm}\label{thm:nilpotency}
If the dimension is less than five, any Delaunay decomposition 
is simplicially generating,
in particular, the nilpotency \cite[1.15]{AN99} 
of any Delaunay decomposition 
is equal to one. 
\end{thm}
See \cite[1.14]{AN99} and \cite[1.6]{Nakamura99}.  
\begin{proof}Clear from Corollary~\ref{cor:generation} and 
Subsec.~\ref{subsec:Voronoi cones}~(i)-(iii). 
\end{proof}

\begin{cor}Any projectively stable quasi-abelian scheme (abbr. PSQAS) 
over a field is reduced if its dimension is less than 5. 
\end{cor}
\begin{proof}
 Let $Q_0$ be a $g$-dimensional PSQAS over $k$. 
See \cite{Nakamura99}. 
It suffices to prove the corollary in the case where 
that it is a special fiber of a {\em totally degenerate} 
flat projective family $(Q,\cL_Q)$ over $\Spec R$ for some complete discrete valuation ring $R$ with residue field $k$. The family $(Q,\cL_Q)$ is given by \cite[5.8]{Nakamura99}. Let $(P,\cL)$ be the normalization of $(Q,\cL_Q)$, and 
$\cL$ the pullback of $\cL_Q$ to $P$ \cite[3.1, 5.1]{Nakamura99}. The family 
$(Q,\cL_Q)$ (resp. $(P,\cL)$) is an algebraization of the formal quotient of the formal completion of an $R$-scheme $(\wQ,\wcL_Q)$ (resp. $(\wP,\wcL)$) locally of finite type, each of which is given in \cite[5.3]{Nakamura99} (resp. a bit implicitly in \cite[3.4, 3.7, 3.8, 4.10]{Nakamura99}). The local affine chart 
$U(c)$ of $\wP$ 
$(c\in X)$ is the normalization of a local affine chart 
$W(c)$ of $\wQ$, where $X$ is a lattice of rank $g$.  Let 
$R(c)$ (resp. $S(c)$) be  
the coordinate ring of $U(c)$ (resp. $W(c)$), where each $R(c)$ (resp. $S(c)$) 
is isomorphic to $R(0)$ (resp. $S(0)$). Moreover there are some monomials 
$\zeta_{x,0}$ and $\xi_{x,0}$ such that 
\begin{align*}
R(0)&=R[\zeta_{x,0}, x\in C(0,\sigma)\cap X, \sigma\in\Del(0)],\\
S(0)&=R[\xi_{x,0}, x\in\Semi(0,\sigma), \sigma\in\Del(0)], 
\end{align*}where  by \cite[Definition~3.4]{Nakamura99}, 
 $\xi_{x,0}=\zeta_{x,0}$ if $x\in\sigma\cap X$ for $\sigma\in\Del(0)$. \par
By Theorem~\ref{thm:nilpotency}, $C(0,\sigma)\cap X=\Semi(0,\sigma)$ for any 
$g$-dimensional Delaunay cell $\sigma$ if $g\leq 4$. It follows that $\wQ=\wP$, $Q=P$ and $Q_0=P_0$.\par
 It follows from \cite[3.12]{AN99} that 
$P_0$ is generically reduced (along any irreducible component of it) 
if $g\leq 4$. Hence there is no nontrivial torsion of the structure sheaf 
$O_{P_0}$ whose support contains an irreducible component of $P_0$. 
Since $P$ is Cohen-Macaulay, so is $P_0$ \cite[4.1]{AN99}, hence 
there is no torsion of  
$O_{P_0}$ whose support is nonempty but at most $(g-1)$-dimensional. 
This completes the proof.
\end{proof}

\section{Degree two and three}
\subsection{Degree two}
\label{subsec:degree two}
Let $s_1$ and $s_2$ be a basis of $X$, $s_{12}:=s_1+s_2$ 
and $Y_2$ the space of real symmetric 
positive semi-definite bilinear forms on $\XR$, each bilinear form being 
identified with a $2\times 2$ matrix.  
We define $V_1$, $V_2$ and $V_1\cap V_2$ as follows:
\begin{align*}
V_1&=\bR_0 e_{13}+\bR_0 e_{23}+\bR_0 e_{12},\\
V_2&=\bR_0 e_{13}+\bR_0 e_{23}+\bR_0 f_{12},\\
V_1\cap V_2&=\bR_0 e_{13}+\bR_0 e_{23},
\end{align*}where 
\begin{gather*}
e_{13}=\begin{pmatrix}
1&0\\
0&1
\end{pmatrix},\ 
e_{23}=\begin{pmatrix}
0&0\\
0&1
\end{pmatrix},\
e_{12}=\begin{pmatrix}
1&-1\\
-1&1
\end{pmatrix},\
f_{12}=\begin{pmatrix}
1&1\\
1&1
\end{pmatrix}.
\end{gather*} 
Then $V_i$ $(i=1,2)$ and $V_1\cap V_2$ are Voronoi cones, 
each of whose relative interior determines 
a unique Delaunay decomposition of $\XR=\bR^2$:
\begin{align*}
\Del_{V_1}&=\{\sigma_1, \sigma_2\ \text{their translates by $X$ 
and their faces}\},\\
\Del_{V_2}&=\{\sigma_3, \sigma_4\ \text{their translates by $X$ 
and their faces}\},\\
\Del_{V_1\cap V_2}&=\{\sigma_5\ \text{their translates by $X$ 
and their faces}\}
\end{align*}
where 
\begin{gather*}
\sigma_{1}=\langle 0,s_1,s_{12}\rangle,\ 
\sigma_{2}=\langle 0,s_2,s_{12}\rangle,\\
\sigma_{3}=\langle 0,s_1,s_2\rangle,\ 
\sigma_{4}=\langle s_1,s_2,s_{12}\rangle,\\
\sigma_{5}=\langle 0,s_1,s_2,s_{12}\rangle.
\end{gather*}

Let $B\in V_1^0$ and let $c_i$ be the center of $\sigma_i$ $(i=1,2)$: 
\begin{gather*}
B(c_1,c_1)=B(s_1-c_1.s_1-c_1)=B(s_{12}-c_1,s_{12}-c_1),\\
B(c_2,c_2)=B(s_2-c_2.s_2-c_2)=B(s_{12}-c_2,s_{12}-c_2).
\end{gather*}Equivalently, 
\begin{align*}
B(s_1,s_1)&=2B(s_1,c_1)\\
B(s_2,s_2)+2B(s_1,s_2)&=2B(s_2,c_1),\\
B(s_2.s_2)&=2B(s_2,c_2)\\
B(s_1,s_1)+2B(s_1,s_2)&=2B(s_1,c_2).
\end{align*}

Let $V_1^0$ be the relative interior of $V_1$, $t:=b_{12}$, 
$B(t):=(b_{ij})\in V_1^0$,  
and consider the limit $t\to 0$. 
Let $B(0)=\lim B(t)=\diag(b_{11},b_{22})$. 
Since $B(0)(s_1,s_2)=0$, $B(0)(\bullet,\lim c_1)=B(0)(\bullet,\lim c_2)$, hence $\lim c_1=\lim c_2$. This implies that $\sigma_1$ and $\sigma_2$ 
fuse together into a D-cell of $B(0)$:
$$\sigma_5=\sigma_1\cup\sigma_2.$$ 

In other words, the Delaunay decomposition of $\lim B$ is obtained 
as a suitable fusion of 
the Delaunay decomposition of $B$. \par
Similarly as we take the limit from $V_2$ to $V_1\cap V_2$, we have another fusion 
$$\sigma_5=\sigma_3\cup\sigma_4.
$$

\begin{rem}
\label{rem:sigma4 does not contain 0}Note that 
$C(0,\sigma_5)=C(0,\sigma_3)$ and $0\not\in\sigma_4$. 
\end{rem}

\subsection{Degree three}
There is, up to $\GL(3,\bZ)$-equivalence, 
 a unique Voronoi cone of dimension 6 in $Y_3$:
$$V=\Sigma_{1\leq i<j\leq 4}\bR_0e_{ij}.
$$

The Delaunay decomposition of $V$ consists of 
3 dimensional cells
\begin{gather*}
\sigma_{ijk}:=\langle 0,s_i,s_{ij},s_{ijk}\rangle,
\ 1\leq i,j,k\leq 3,\ \text{all distinct}\\
\text{and $\GL(3,\bZ)$-transforms and $\bZ^3$-translates.}
\end{gather*}

\section{Degree four}

\subsection{Notation}
\label{subsec:notation}
We define 
\begin{gather*}
e_{ij}=\begin{pmatrix}0&0&0&0\\
0&\overset{\underset{\smile}{i}}{1}&\overset{\underset{\smile}{j}}{-1}&0\\
0&-1&1&0\\
0&0&0&0
\end{pmatrix}\begin{matrix}\phantom{0}\\
\overset{\phantom{\underset{\smile}{i}}}{(i}\\
(j\\
\phantom{0}\end{matrix},\quad 
e_{i5}=\begin{pmatrix}0&0&0&0\\
0&\overset{\underset{\smile}{i}}{1}&0&0\\
0&0&0&0\\
0&0&0&0
\end{pmatrix}\begin{matrix}\phantom{0}\\
\overset{\phantom{\underset{\smile}{i}}}{(i}\\
\phantom{0}\\
\phantom{0}\end{matrix}\\ 
e_{12,345}=\begin{pmatrix}2&1&-1&-1\\
1&2&-1&-1\\
-1&-1&2&0\\
-1&-1&0&2
\end{pmatrix},\ 
f_{1234}=\begin{pmatrix}1&1&-1&-1\\
1&1&-1&-1\\
-1&-1&1&1\\
-1&-1&1&1
\end{pmatrix}\\
g_{123}=\begin{pmatrix}
1&1&-1&0\\
1&1&-1&0\\
-1&-1&1&0\\
0&0&0&0
\end{pmatrix},\ 
g_{124}=\begin{pmatrix}
1&1&0&-1\\
1&1&0&-1\\
0&0&0&0\\
-1&-1&0&1
\end{pmatrix}
\end{gather*}where 
$$\omega=e_{12,345}=\frac{1}{3}(\sum_{(i,j)\neq (1,2)}e_{ij}+f_{1234}+g_{123}+g_{124}).
$$

We also define
\begin{align*}
V_1&=\Sigma_{i,j}\bR_{0}e_{ij},\\
V_1\cap V_2&=\Sigma_{(i,j)\neq (12)}\bR_{0}e_{ij}\\
V_2&=\Sigma_{(i,j)\neq (1,2)}\bR_{0}e_{ij}+\bR_{0}e_{12345},\\
V_2\cap V_3&=\Sigma_{(i,j)\neq (1,2), (3,4)}\bR_{0}e_{ij}
+\bR_{0}e_{12345},\\
V_3&=\Sigma_{(i,j)\neq (1,2), (3,4)}\bR_{0}e_{ij}
+\bR_{0}e_{12345}+\bR_{0}f_{1234}\\
W_0&=\Sigma_{(i,j)\neq (1,2), (3,4)}\bR_{0}e_{ij}
+\bR_{0}f_{1234}\\
K&=\Sigma_{(i,j)\neq (1,2)}\bR_{0}e_{ij}+\bR_{0}g_{123}
+\bR_{0}g_{124}+\bR_{0}f_{1234},
\end{align*}where $V_i\subset K$ $(i=2,3)$. 
It is clear that $V_i$ is a 10-dimensional cone, which is  
the same as the Delaunay triangulation $\cD_i$ 
in the sense of \cite[p.~51]{Val03}.

\begin{rem}With the notation in \cite[p.~234]{Igusa67}, 
we define the chambers
\begin{align*}
F_{ab}:&=\Sigma_{(i,j)\neq (a,b)}\bR_{0}e_{ij}+\bR_{0}e_{abcde},\\
G_{abcd}:&=\Sigma_{(i,j)\neq (a,b),(c,d)}\bR_{0}e_{ij}+\bR_{0}e_{abcde}
+\bR_{0}e_{cdabe}\\
F_{abcd}&=\{Y=(y_{ij})\in G_{abcd};y_{ab}\geq y_{cd}\},\\
F_{cdab}&=\{Y=(y_{ij})\in G_{abcd};y_{cd}\geq y_{ab}\}
\end{align*}for any subset $\{a,b\}\subset [1,4]$ and 
$\{c,d\}=[1,4]\setminus\{a,b\}$. This supplements 
the definition of $F_{ab}$ and $F_{abcd}$ 
in \cite[p.~664]{Nakamura99}.

Note that $F_{12}=V_2$ is a Voronoi cone. Another chamber $G_{1234}$ is the union of two $\GL(4,\bZ)$-equivalent Voronoi cones $V_3$ and $V_4$ (of type III) where  
\begin{gather*}
V_3=\Sigma_{(i,j)\neq (1,2), (3,4)}\bR_{0}e_{ij}
+\bR_{0}e_{12345}+\bR_{0}f_{1234}=F_{1234},\\
V_4:=\Sigma_{(i,j)\neq (1,2), (3,4)}\bR_{0}e_{ij}
+\bR_{0}e_{34125}+\bR_{0}f_{1234}=F_{3412},\\
V_3\cap V_4=\Sigma_{(i,j)\neq (1,2), (3,4)}\bR_{0}e_{ij}+\bR_{0}f_{1234}=W_0.
\end{gather*}because 
\begin{equation*}e_{12345}+e_{34125}
=f_{1234}+\sum_{(i,j)\neq (1,2),(3,4)}e_{ij}.
\end{equation*}
\end{rem}

\subsection{Black forks and red triangles}
\label{subsec:BF and RT}
In what follows we identify the space of  
real symmetric bilinear forms on $\bR^4$ 
with the space of real quadratic polynomials 
of 4 variables in an obvious way.  There are two perfect cones $V_1$ and $K$ 
in 4 variables by \cite[Part I, p.~172]{Voronoi09}. 
The cone $V_1$ is a 10-dimensinal cone generated by 
\begin{gather*}
x_i^2,  (x_i-x_j)^2,  (i, j\in[1,4])
\end{gather*}where $x_i^2$ (resp. $(x_i-x_j)^2$) corresponds 
to $e_{i5}$ (resp. $e_{ij}$).    \par
The cone  
$K$ is also a 10-dimensional cone generated by 
\begin{equation}
\begin{aligned}\label{eq:generators of K}
&x_i^2,  (x_i-x_j)^2,  (i, j\in[1,4],(i,j)\neq (1,2)), \\
(x_1+&x_2-x_3)^2, (x_1+x_2-x_4)^2, (x_1+x_2-x_3-x_4)^2.
\end{aligned}
\end{equation}

Let $\omega$ be  
\begin{align*}
\omega&=\frac{1}{3}(\sum_{i=1}^4x_i^2+\sum_{(i,j)\neq (1,2)}(x_i-x_j)^2\\
&+(x_1+x_2-x_3)^2+(x_1+x_2-x_4)^2+(x_1+x_2-x_3-x_4)^2)\\
&=2\sum_{i=1}^4x_i^2+2x_1x_2-2\sum_{i=1,2;j=3,4}x_ix_j.
\end{align*}
Note that $\omega=e_{12345}$ in matrix form.

By the transformation (which we refer to as {\em Voronoi transformation})
\begin{equation}
\label{eq:Voronoi transf}
(x_1,x_2,x_3,x_4)\mapsto (x_1+x_2, x_1-x_2, x_1-x_3, x_1-x_4)
\end{equation}
 every generator (\ref{eq:generators of K}) of $K$ is transformed into 
either $(x_i+x_j)^2$ or $(x_i-x_j)^2$. while $(x_1-x_2)^2$ is transformed into $4x_2^2$. 
Hence 
\begin{equation*}
K=\{x=\sum_{i<j}\beta_{ij}(x_i+x_j)^2+\sum_{i<j}\rho_{ij}(x_i-x_j)^2; 
\beta_{ij}\geq 0, \rho_{ij}\geq 0\}.
\end{equation*}

Let $W$ be a face of $K$ of codimension one. 
By \cite[p.~1060]{ER88}, there exist precisely three pairs $(i,j)$ 
 such that $x\in W$ iff either 
$\beta_{ij}=0$ or $\rho_{ij}=0$. 
If $\beta_{pq}=0$  (resp. $\rho_{pq}=0$), we connect the vertices $p$ and $q$ by a black edge (resp. a red edge). Then 
\begin{enumerate}
\item[(i)]
to each face $W$, we can associate a connected graph $\Gamma(W)$ with three colored edges, which is  
either forked or triangular. 
\item[(ii)]
Conversely, for any such graph $\Gamma$ 
there is a unique face $W$ of $K$ of codimension one such that $\Gamma(W)=\Gamma$ (as colored graphs). 
\item[(iii)]
There are 32 forked graphs and 32 triangular graphs by 
\cite[2.1, p.~1061]{ER88}. (Easy)
\item[(iv)]
There are two $\GL(4,\bZ)$-equivalence classes of 64 faces, one being the equivalence class of black forked faces, and the other being the equivalence class of red triangular faces. We call the first equivalence class (resp. the second) BF (resp. RT). 
\item[(v)]
By \cite[2.5, p.~1063]{ER88} any cone generated by $e_{12345}$ and BF (resp. $e_{12345}$ and RT) is a domain of type II 
(resp. type III) in the sense of Voronoi \cite{Voronoi09}. 
\item[(vi)]
By \cite[2.2, 2.3, p.~1060]{ER88}
BF (resp. RT) consists of 48 faces (resp. 16 faces), 
and each equivalence class is an orbit of a group $G$ of order 1152, where 
the group $G$ is generated by three types of elements:
\begin{align*}
\tau_i&:x_i\mapsto -x_i,\ x_j\mapsto x_j\ (j\neq i)\\
\tau_{ij}&:x_i\mapsto x_j,\ x_j\mapsto x_i,\ x_k\mapsto x_k\ (k\neq i,j)\\
\sigma&:x_i\mapsto -x_i+\frac{1}{2}\sum_{k=1}^4x_k,
\end{align*}where $i,j,k\in[1,4]$. 
\end{enumerate}


\subsection{The faces $V_1\cap V_2$ and $V_3\cap V_4$}
The following lemma follows easily from \cite[\S~2, pp.1060-1064]{ER88}.  
\begin{lemma}The cone $W_0=V_3\cap V_4$ is $\GL(4,\bZ)$-equivalent 
to a red triangle face of $K$. Moreover 
$V_1\cap V_2$ is a black triangle face of $K$, hence 
$\GL(4,\bZ)$-equivalent to a black fork face.
\end{lemma}
\begin{proof}
It is obvious from Subsec.~\ref{subsec:BF and RT} 
that $V_1\cap V_2$ is a black triangle with vertices $1,3,4$. Hence it is $\GL(4,\bZ)$-equivalent to a black fork (BF) by 
\cite[\S~2, p.1062]{ER88}. 
On the other hand we note that  
$W_0$ ia a red triangle face (RT).  Indeed, 
after Voronoi transformation (\ref{eq:Voronoi transf}), 
$W_0$ is spanned by 
\begin{gather*}
(x_1+x_2)^2, (x_2+x_3)^2, (x_2+x_4)^2, (x_3+x_4)^2, \\
(x_1-x_2)^2, (x_1-x_3)^2, (x_1-x_4)^2, (x_2-x_3)^2, (x_2-x_4)^2, 
\end{gather*} Hence the missing terms are 
$$(x_1+x_3)^2, (x_1+x_4)^2, (x_3-x_4)^2.$$ 
Therefore the graph of $W_0$ is a triangle 
with two black edges and a red edge.  
The transformation $\tau_1:x_1\mapsto -x_1, x_j\mapsto x_j$ 
$(j\neq 1)$ belongs to the group $G$, by which the missing terms in $W_0$
are transformed into 
$$(x_1-x_3)^2, (x_1-x_4)^2, (x_3-x_4)^2,$$ whence 
the graph of $\tau_1(W_0)$ is a red triangle. 
Hence the face $W_0$ is $\GL(4,\bZ)$-equivalent to a red triangle. 
The rest is clear. 
\end{proof}

\begin{cor}
The cone $V_2$ (resp. $V_3$) is 
a 10-dimensional cone of type II (resp. type III) in the sense of Voronoi 
\cite{Voronoi09}. 
\end{cor}
\begin{proof}Since $V_3$ is 
generated by $W_0$ and $e_{12345}$ and $W_0$ is RT, $V_3$ 
is of type III by \cite[2.5, p.~1063]{ER88}. Meanwhile 
since $V_2$ is generated by $V_1\cap V_2$ and $e_{12345}$, and 
$V_1\cap V_2$ is BF, $V_2$ is of type II.
\end{proof}

\section{The Delaunay decompositions in dimension 4}
\label{sec:Delaunay in dim 4}

\subsection{The Delaunay decomposition of $V_1$}
For $\{a,b,c,d\}=\{1,2,3,4\}$ we define 
\begin{gather*}
\sigma_{abcd}
:=\langle 0,s_a,s_{ab},s_{abc}, s_{1234}\rangle.
\end{gather*}

It is easy to see that the Delaunay decomposition of $V_1$ is 
given by 
\begin{gather*}
\sigma_{abcd},\ 
\text{their faces and their translates by $\bZ^4$.}
\end{gather*}

Let $B\in V_1^0$, $\sigma=\sigma_{1234}$ and $c$ 
the center of the D-cell $\sigma$. Then we have 
\begin{equation*}
B(e-c,e-c)=B(c,c)\ (\forall e\in \sigma\cap X),
\end{equation*} 
hence 
\begin{align*}
B(s_1,s_1)&=2B(s_1,c),\\
B(s_2,s_2)+2B(s_1,s_2)&=2B(s_2,c),\\
B(s_3,s_3)+2B(s_{12},s_3)&=2B(s_3,c),\\
B(s_4,s_4)+2B(s_{123},s_4)&=2B(s_4,c),\\
B(e,e)-2B(e,c)&>0\ (\text{otherwise}).
\end{align*} 

If $B(s_1,s_2)=0$, then the second equation reduces to 
$B(s_2,s_2)=2B(s_2,c)$,  which implies that 
$\langle\sigma,s_2\rangle$ is in the same D-cell of $B\in V_1\cap V_2$.  

\subsection{The fusion of D-cells of $V_1$}
Let $B_0\in (V_1\cap V_2)^0$. Since $B_0\in Y_4^{+}$, $B_0$ determines a 
Delaunay decomposition of $\XR$. Then there exists 
a unique D-cell $\sigma$ of $B_0$ such that 
$\sigma_{1234}\subset\sigma$. Let $c$ be 
the center of $\sigma$. 
Hence 
\begin{equation}\label{eq:fusion sigma1234 V1V2}
\begin{aligned}
B_0(s_1,s_1)&=2B_0(s_1,c_0),\\
B_0(s_2,s_2)&=2B_0(s_2,c_0),\\
B_0(s_3,s_3)+2B_0(s_{12},s_3)&=2B_0(s_3,c_0),\\
B_0(s_4,s_4)+2B_0(s_{123},s_4)&=2B_0(s_4,c_0),\\
B_0(e,e)-2B_0(e,c_0)&\geq 0\ (\text{otherwise}).
\end{aligned} 
\end{equation}
\begin{lemma}Let $\sigma$ be the unique D-cell $\sigma$ of $V_1\cap V_2$ 
such that 
$\sigma_{1234}\subset\sigma$.  Then 
$\sigma=\sigma_{1234}\cup \sigma_{2134}$. 
Similarly  $\sigma_{1243}\cup\sigma_{2143}$ is 
also a D-cell of $V_1\cap V_2$.
\end{lemma}
\begin{proof}
By definition there exists $r_{ij}>0$ such that 
\begin{align*}
B_0&=\sum_{(i,j)\neq (1,2)}r_{ij}e_{ij}.
\end{align*}
It easy to see that 
if $x=\sum_{i=1}^4x_is_i$, then by (\ref{eq:fusion sigma1234 V1V2})
\begin{align*}
&B_0(x,x)-2B_0(x,c_0)\\
=&\sum_{i=1}^4r_{i5}(x_i^2-x_i)+\sum_{
\tiny
\begin{matrix}
(ij)\neq (12)\\
1\leq i<j\leq 4
\end{matrix}}
r_{ij}((x_i-x_j)^2-(x_i-x_j))\geq 0.
\end{align*}

Since every $r_{ij}>0$ for $(ij)\neq(12)$, 
$B_0(x,x)-2B_0(x,c_0)=0$ iff every term in the rhs is equal to $0$:
\begin{gather*}
x_i^2-x_i=0, (x_i-x_j)^2-(x_i-x_j)=0\ \forall i,j.
\end{gather*} Hence $x_i=0,1$ and $x_i-x_j=0,1$ for $i<j$ $(i,j)\neq (1,2)$. 
It follows that $B_0(x,x)-2B_0(x,c_0)=0$ iff
\begin{gather*}
x=0, (1,0,0,0), (0,1,0,0), (1,1,0,0), (1,1,1,0), (1,1,1,1).
\end{gather*}

This shows $\sigma=\sigma_{1234}\cup\sigma_{2134}$. 
In summary, $\sigma_{1234}$ and $\sigma_{2134}$ 
fuse together into a unique D-cell of $V_1\cap V_2$ 
when $B\in V_1^0$ approaches a point $B_0$ of $V_1\cap V_2$. 
Similarly $\sigma_{1243}\cup\sigma_{2143}$ is a D-cell of $V_1\cap V_2$.   
\end{proof}

\subsection{Tables 1 and 2}
\label{subsec:tables 1 and 2}

By similar computations, we obtain Tables 1 and 2. 
Let us explain what the tables show. The D-cells $\sigma_{1234}$ and $\sigma_{2134}$ fuse together into a D-cell of $V_1\cap V_2$, which divides into two D-cells of $V_2$
$$[0,s_1,s_2,s_{123},s_{1234}], [s_1,s_2,s_{12},s_{123},s_{1234}], 
$$which are no longer D-cells of $V_1$. These are numbered 1 and 2 respectively. 
The D-cells numbered from 3 to 12 are understood similarly. The D-cells of $V_1$ numbered from 13 to 24 are 
D-cells of both $V_1\cap V_2$ and $V_2$. This is what Table 1 shows.\par 
Next the D-cells of $V_2$ numbered 2 and 4 fuse together into a D-cell of 
$V_2\cap V_3$, which decomposes into two D-cells of $V_3$:
$$[s_1,s_2,s_{12},s_{123},s_{124}], [s_1,s_2,s_{123},s_{124},s_{1234}]
$$ Similarly the D-cells of $V_2$ numbered 9 and 11 (resp. 17 and 18, 19 and 20)
fuse together into a  D-cell of $V_2\cap V_3$, which divides into two D-cells of $V_3$ which are no longer D-cells of $V_2$. However the other D-cells of 
$V_2$ are also D-cells of both $V_2\cap V_3$ and $V_3$. This is what Table 2 shows.  

The following is clear from Tables:
\begin{thm}
\label{thm:Vor dim 4}{\rm (\cite{Voronoi09})}\ 
The Delaunay decomposition of $V_i$ $(i=1,2,3)$ consists of $4$-dimensional integral simplicies and their faces, where we mean by a $4$-dimensional integral simplex a convex closure $\langle a_0,\cdots, a_4\rangle$ such that 
$a_i\in X$ and $a_i-a_0$ is a $\bZ$-basis of $X$. In particular, the Delaunay decomposition of $V_i$ is simplicially generating.
\end{thm}

\begin{table}[ht]
 \begin{centering}
 \renewcommand{\arraycolsep}{0.6em}
 \renewcommand{\arraystretch}{1.4}
 $\begin{array}{|l|l|l|c|}\hline
V_1&V_1\cap V_2&V_2&\text{no}\\
\hline
\sigma_{1234}&\sigma_{1234}\cup\sigma_{2134}&\langle 0,s_1,s_2,s_{123},s_{1234}\rangle &1\\
\sigma_{2134}&&\langle s_1,s_2,s_{12},s_{123},s_{1234}\rangle &2\\
\sigma_{1243}&\sigma_{1243}\cup\sigma_{2143}&\langle 0,s_1,s_2,s_{124},s_{1234}\rangle&3\\
\sigma_{2143}&&\langle s_1,s_2,s_{12},s_{124},s_{1234}\rangle&4\\
\sigma_{3124}&\sigma_{3124}\cup\sigma_{3214}&\langle 0,s_3,s_{13},s_{23},s_{1234}\rangle&5\\
\sigma_{3214}&&\langle 0,s_{13},s_{23},s_{123},s_{1234}\rangle&6\\
\sigma_{4123}&\sigma_{4123}\cup\sigma_{4213}&\langle 0,s_4,s_{14},s_{24},s_{1234}\rangle&7\\
\sigma_{4213}&&\langle 0,s_{14},s_{24},s_{124},s_{1234}\rangle&8\\
\sigma_{3412}&\sigma_{3412}\cup\sigma_{3421}&\langle 0,s_3,s_{34},s_{134},s_{234}\rangle&9\\
\sigma_{3421}&&\langle 0,s_3,s_{134},s_{234},s_{1234}\rangle&10\\
\sigma_{4312}&\sigma_{4312}\cup\sigma_{4321}
&\langle 0,s_4,s_{34},s_{134},s_{234}\rangle&11\\
\sigma_{4321}&&\langle 0,s_4,s_{134},s_{234},s_{1234}\rangle&12\\
\sigma_{1324}&\sigma_{1324}&\sigma_{1324}&13\\
\sigma_{1423}&\sigma_{1423}&\sigma_{1423}&14\\
\sigma_{2314}&\sigma_{2314}&\sigma_{2314}&15\\
\sigma_{2413}&\sigma_{2413}&\sigma_{2413}&16\\
\sigma_{1342}&\sigma_{1342}&\sigma_{1342}&17\\
\sigma_{1432}&\sigma_{1432}&\sigma_{1432}&18\\
\sigma_{2341}&\sigma_{2341}&\sigma_{2341}&19\\
\sigma_{2431}&\sigma_{2431}&\sigma_{2431}&20\\
\sigma_{3142}&\sigma_{3142}&\sigma_{3142}&21\\
\sigma_{4132}&\sigma_{4132}&\sigma_{4132}&22\\
\sigma_{3241}&\sigma_{3241}&\sigma_{3241}&23\\
\sigma_{4231}&\sigma_{4231}&\sigma_{4231}&24\\
\hline
 \end{array}$
 \caption{Fusion and division (1)}
 \label{Tab_Fusion_Decomp}
 \end{centering}
 \end{table}

\newpage
 \begin{table}[ht]
 \begin{centering}
 \renewcommand{\arraycolsep}{0.6em}
 \renewcommand{\arraystretch}{1.4}
 $\begin{array}{|c|l|l|l|}\hline
\text{no}&V_2&V_2\cap V_3&V_3\\
\hline
2&\langle s_1,s_2,s_{12},s_{123},s_{1234}\rangle
&\langle s_1,s_2,s_{12},s_{123},s_{1234}\rangle
&\langle s_1,s_2,s_{12},s_{123},s_{124}\rangle\\
4&\langle s_1,s_2,s_{12},s_{124},s_{1234}\rangle
&\cup \langle s_1,s_2,s_{12},s_{124},s_{1234}\rangle&\langle s_1,s_2,s_{123},s_{124},s_{1234}\rangle\\
9&\langle 0,s_3,s_{34},s_{134},s_{234}\rangle&\langle 0,s_3,s_{34},s_{134},s_{234}\rangle&\langle 0,s_3,s_4,s_{134},s_{234}\rangle\\
11&\langle 0,s_4,s_{34},s_{134},s_{234}\rangle&\cup \langle 0,s_4,s_{34},s_{134},s_{234}\rangle&\langle s_3,s_4,s_{34},s_{134},s_{234}\rangle\\
17&\sigma_{1342}&\sigma_{1342}\cup\sigma_{1432}&\langle 0,s_1,s_{13},s_{14},s_{1234}\rangle\\
18&\sigma_{1432}&&\langle 0,s_{13},s_{14},s_{134},s_{1234}\rangle\\
19&\sigma_{2341}&\sigma_{2341}\cup\sigma_{2431}&\langle 0,s_2,s_{23},s_{24},,s_{1234}\rangle\\
20&\sigma_{2431}&&\langle 0,s_{23},s_{24},s_{234},s_{1234}\rangle\\
1&\langle 0,s_1,s_2,s_{123},s_{1234}\rangle
&\langle 0,s_1,s_2,s_{123},s_{1234}\rangle
&\langle 0,s_1,s_2,s_{123},s_{1234}\rangle
\\%
3&\langle 0,s_1,s_2,s_{124},s_{1234}\rangle&\langle 0,s_1,s_2,s_{124},s_{1234}\rangle&\langle 0,s_1,s_2,s_{124},s_{1234}\rangle\\
5&\langle 0,s_3,s_{13},s_{23},s_{1234}\rangle&\langle 0,s_3,s_{13},s_{23},s_{1234}\rangle&\langle 0,s_3,s_{13},s_{23},s_{1234}\rangle\\
6&\langle 0,s_{13},s_{23},s_{123},s_{1234}\rangle&\langle 0,s_{13},s_{23},s_{123},s_{1234}\rangle&\langle 0,s_{13},s_{23},s_{123},s_{1234}\rangle\\
7&\langle 0,s_4,s_{14},s_{24},s_{1234}\rangle&\langle 0,s_4,s_{14},s_{24},s_{1234}\rangle&\langle 0,s_4,s_{14},s_{24},s_{1234}\rangle\\
8&\langle 0,s_{14},s_{24},s_{124},s_{1234}\rangle&\langle 0,s_{14},s_{24},s_{124},s_{1234}\rangle&\langle 0,s_{14},s_{24},s_{124},s_{1234}\rangle\\
10&\langle 0,s_3,s_{134},s_{234},s_{1234}\rangle&\langle 0,s_3,s_{134},s_{234},s_{1234}\rangle&\langle 0,s_3,s_{134},s_{234},s_{1234}\rangle\\
12&\langle 0,s_4,s_{134},s_{234},s_{1234}\rangle&\langle 0,s_4,s_{134},s_{234},s_{1234}\rangle&\langle 0,s_4,s_{134},s_{234},s_{1234}\rangle\\
13&\sigma_{1324}&\sigma_{1324}&\sigma_{1324}\\
14&\sigma_{1423}&\sigma_{1423}&\sigma_{1423}\\
15&\sigma_{2314}&\sigma_{2314}&\sigma_{2314}\\
16&\sigma_{2413}&\sigma_{2413}&\sigma_{2413}\\
21&\sigma_{3142}&\sigma_{3142}&\sigma_{3142}\\
22&\sigma_{4132}&\sigma_{4132}&\sigma_{4132}\\
23&\sigma_{3241}&\sigma_{3241}&\sigma_{3241}\\
24&\sigma_{4231}&\sigma_{4231}&\sigma_{4231}\\
\hline
 \end{array}$
 \caption{Fusion and division (2)}
 \label{Tab_Fusion_Decomp 2}
 \end{centering}
 \end{table}

\end{document}